\documentclass[a4paper,11pt]{article}

\usepackage[affil-it]{authblk}

\usepackage{indentfirst}
\usepackage[margin = 1 in]{geometry}

\usepackage[numbers,sort&compress]{natbib}
\usepackage{amssymb}
\usepackage{amsmath}
\usepackage{stmaryrd}
\usepackage{empheq}
\usepackage{amsthm}
\usepackage{mathrsfs}
\usepackage{euscript}
\usepackage{soul}
\usepackage{xcolor}
\usepackage[hypertexnames=false,colorlinks,linkcolor={blue},citecolor={blue}]{hyperref}
\usepackage{graphicx}
\usepackage{float}
\usepackage{caption}
\usepackage{epsfig}
\usepackage{subcaption}
\usepackage{scalerel} % used for decreasing the size of \bigcap and \bigcup

\usepackage{enumitem}
\setlist{
  leftmargin=*, 
  topsep = 4pt, 
  itemsep = 1pt, 
}

\usepackage[T1]{fontenc}
\usepackage{lmodern}
\usepackage{microtype}
\usepackage[english]{babel}

\usepackage{fancyhdr}
\fancypagestyle{draftfirstpage}{
  
  \fancyhf{}
}
\numberwithin{equation}{section}

\newtheorem{theorem}{Theorem}[section]
\newtheorem{corollary}[theorem]{Corollary}
\newtheorem{lemma}[theorem]{Lemma}
\newtheorem{proposition}[theorem]{Proposition}
\theoremstyle{definition}
\newtheorem{definition}[theorem]{Definition}

\newtheorem{remark}[theorem]{Remark}
\newtheorem{example}[theorem]{Example}

\allowdisplaybreaks

\newcommand{\eg}{{\itshape e.g., }}
\newcommand{\eqfinp}{\,.}
\newcommand{\eqfinv}{\,,}
\newcommand{\eqsepv}{\,,}

\newcommand{\eqsepand}{\quad \text{and} \quad}

\usepackage{macros}

\usepackage{array} 

\makeatletter
\let\original@l@section\l@section
\renewcommand*{\l@section}[2]{{\let\addvspace\@gobble\original@l@section{#1}{#2}}}
\makeatother

\newcommand{\rev}[1]{\ifmmode\expandafter\mathcolor\else\expandafter\textcolor\fi{blue}{#1}}

\DeclareMathOperator{\Eff}{Eff}
\DeclareMathOperator{\wEff}{wEff}
\DeclareMathOperator{\IEff}{IEff}
\DeclareMathOperator{\ENE}{ENE}
\DeclareMathOperator{\wENE}{wENE}
\DeclareMathOperator{\INE}{INE}
\DeclareMathOperator{\GNE}{GNE}
\title{Stability of ideal Nash equilibrium points in generalized multi-criteria games, interval-valued games, and existence in single-leader--multi-follower problems}

\author[1]{Urairat Deepan}
\author[1,2]{Poom Kumam}
\author[1,2]{Parin Chaipunya\thanks{Corresponding author}}

\affil[1]{
  {Department of Mathematics, Faculty of Science, King Mongkut's University of Technology Thonburi},
  {126 Pracha Uthit Rd.},
  {Bang Mod, Thung Khru, Bangkok},
  {10140}, 
  {Thailand.}
}
\affil[2]{
  {Center of Excellence in Theoretical and Computational Science (TaCS-CoE), King Mongkut's University of Technology Thonburi},
  {126 Pracha Uthit Rd.},
  {Bang Mod, Thung Khru, Bangkok},
  {10140}, 
  {Thailand.}
}

\affil[ ]{
  Emails: urairat.deepan@mail.kmutt.ac.th, poom.kum@kmutt.ac.th, parin.cha@kmutt.ac.th
  \vspace{.25cm}
}

\date{}

\begin{document}
\maketitle 
\thispagestyle{draftfirstpage}

\vspace{-3em}

\begin{abstract}

  We study the stability of ideal Nash equilibria in generalized multi-criteria games, in which each player's vector-valued payoff is ordered by a closed cone and the feasible strategy sets depend on the other players' strategies.
  Under compactness and continuity assumptions, the ideal Nash equilibrium correspondence is upper semicontinuous on the class of games admitting an ideal equilibrium.
  Our main result does not require the ordering cones to be convex or pointed.
  In fact, the ideal equilibria are exactly the common equilibria of associated scalar games when all the cones are convex.
  As consequences of our results, the stability of interval-valued games is derived and existence theorems for single-leader-multi-follower problems are proved for both optimistic and pessimistic sense.

  \noindent\textbf{Keywords:} Ideal Nash equilibrium; Generalized multi-criteria game; Vector optimization; Upper semicontinuity; Interval order; Single-leader--multi-follower problem; Pessimistic bilevel problem.

  \noindent\textbf{MSC 2020:} 91A10; 91A65; 90C29; 49J53; 54C60.
\end{abstract}

{
  \footnotesize
  \tableofcontents
}

%================================================================================

\section{Introduction}

% Bilevel optimization models hierarchical decision-making through an upper-level optimization problem constrained by a lower-level problem \cite{orig6,orig8,orig13,orig5}.
% The classical Stackelberg formulation captures sequential decisions by a leader and a follower \cite{orig25,orig26}.
% In a single-leader--single-follower model, the leader first selects a decision and anticipates the follower's response; the resulting response restriction makes the leader's feasible region dependent on the lower-level solution set~\cite{orig3,orig7,orig21,Dempe2002,orig28}.
%
% With multiple followers, one must additionally specify their mode of interaction.
% In a noncooperative formulation, the followers may be modeled as players in a Nash game, each responding to the leader's decision and to the other followers' strategies.
% Cooperative and sequential variants lead to different hierarchical models and solution concepts; see the discussions in \cite{orig19,orig16,orig23,orig24,orig12,orig15,orig25}.
% When the followers' game has multiple equilibria, an \emph{optimistic} leader chooses a feasible leader--follower pair minimizing its objective, whereas a pessimistic formulation evaluates an adverse follower response~\cite{orig4,LoridanMorgan1989}.
% Computational reformulations of bilevel and equilibrium-constrained problems may be nonconvex or otherwise difficult; the relevant complexity depends on the chosen problem class~\cite{orig4,orig15,orig17}.

In multi-criteria games, a player's objective has several components, so payoffs are compared through a cone-induced order rather than a single scalar ranking \cite{Luc1989}.
Games with vector payoffs go back to~\citet{Blackwell1956} and~\citet{ShapleyRigby1959}, and solution concepts and the existence of Pareto and weakly Pareto equilibria have since been studied extensively~\cite{Corley1985,GhosePrasad1989,Zhao1991,Wang1993,YuYuan1998}.
See also~\citet{Bade2005} for Nash equilibria under incomplete preferences and~\citet{CrespiKuroiwaRocca2025} for a recent characterization of equilibria in vector-valued matrix games.
The vector components are usually ranked through the concepts of efficient, weakly efficient, and ideal efficient points, and they represent different levels of optimality.
These concepts are then used to define the equilibrium points of a multi-criteria game.
While a (weakly) efficient point requires that no feasible point (weakly) dominates it, an ideal efficient point demands that it must dominate \emph{every} feasible alternatives.

More precisely, let~$N=\{1,\ldots,\ell\}$ be the set of players, let~$X=\prod_{\nu\in N}X_{\nu}$, and let~$K_{\nu}:X_{-\nu} \multimap X_{\nu}$ denote player~$\nu$'s constraint map.
With~$u_{\nu}:X\to\R^{d_{\nu}}$ denoting its vector-valued payoff and~$C_{\nu}\subset\R^{d_{\nu}}$ its ordering cone, we employ the \emph{maximization} convention
\begin{align*}
  a\preceq_{C_{\nu}}b\quad\Longleftrightarrow\quad b-a\in C_{\nu}\eqfinp
\end{align*}
An \emph{ideal Nash equilibrium} is a feasible profile~$x^{\ast}\in X$ satisfying, for every~$\nu\in N$,
\begin{align}\label{eq:intro-ine}
  x_{\nu}^{\ast}\in K_{\nu}(x_{-\nu}^{\ast})\eqsepand
  u_{\nu}(w_{\nu},x_{-\nu}^{\ast})\preceq_{C_{\nu}}u_{\nu}(x^{\ast})
  \quad\text{for every }w_{\nu}\in K_{\nu}(x_{-\nu}^{\ast})\eqfinp
\end{align}
Denote this game by~$\Gamma$, then we write~$\INE(\Gamma)$ for the set of ideal equilibria of the corresponding generalized game~$\Gamma$.
The precise description of this problem is given in Definitions~\ref{def:efficiency} and~\ref{def:game}.
For closed convex cones, ideal equilibria coincide with the common equilibria of the scalar games obtained by weighting the payoffs with dual-cone vectors, so part of our stability analysis can also be derived from the scalar counterpart (Proposition~\ref{prop:scalar-reduction}).
Our direct arguments, however, also cover closed nonconvex cones for which no such reduction exists.

The central stability question concerns how the set of ideal Nash equilibria responds to changes in the players' payoff functions and constraint maps.
For scalar games, stability in the sense of essential equilibria and essential components was studied by~\citet{Yu1999,YuLuo1999}, and these ideas were extended to weakly Pareto--Nash equilibria of multiobjective generalized games by~\citet{Yang2003,Lin2005,SongWang2010}.
Set-valued analysis provides closed-graph and semicontinuity tools suited to this purpose \cite{AubinFrankowska2009,RockafellarWets1998}. 
The approach of this paper is also motivated by the generic-continuity viewpoint is related to the essential-equilibrium approach of~\citet{YuYu2006}.

In this study, we formulate the parameter space with uniform payoff convergence and Hausdorff convergence of constraint sets.
Our stability theorem applies to arbitrary closed ordering cones, which need not be convex or pointed, provided ideal equilibria exist in the parameter class under consideration.
In fact, we pointed out explicitly that whehn the cone is convex, we can derive the stability result directly from the known literature of scalar generalized Nash equilibrium problems.
After that we continue deriving by another method the stability in the case that the cone is not necessarily convex.
Because the class of games admitting ideal equilibria is a complete metric space, Fort's theorem \cite{Fort1951} further yields continuity of the equilibrium correspondence on a dense residual subset.
The lower--upper (LU) ordering of interval-valued objectives is also recalled as a special instance of the~$\R_{+}^{2}$ cone order \cite{Wu2007}.
In Subsection~\ref{sec:interval-valued-games} we use this identification to define ideal Nash equilibria of interval-valued games and to derive their stability from our proposed multi-criteria results.

We then use the equilibrium stability result in a single-leader--multi-follower problem.
The leader chooses a decision~$z\in Z_{0}$, which determines the followers' game~$\lambda(z)$, while the followers select an ideal Nash equilibrium~$x\in S(\lambda(z))$.
The existence result uses compactness of the joint feasible set and lower semicontinuity of the leader's objective, in line with the classical Weierstrass principle and optimistic bilevel optimization \cite{orig3,Dempe2002,LoridanMorgan1989}.
For the pessimistic formulation, upper semicontinuity of the equilibrium correspondence shows that the pessimistic value is attained by some leader decision and follower equilibrium.
This decision solves the pessimistic problem when the correspondence is also lower semicontinuous at the limiting parameter, and an example shows that this extra condition cannot be dropped.

The paper is organized as follows.
Section~\ref{sec:preliminaries} introduces the cone order, ideal Nash equilibria, and relevant notions from set-valued analysis.
Section~\ref{sec:stability} proves equilibrium stability and generic continuity and relates the results to the scalar theory when the cones are convex.
Subsection~\ref{sec:interval-valued-games} recalls LU ordering for interval-valued objectives and specializes the stability results to interval-valued games.
Section~\ref{sec:single-leader-multi-follower} applies the stability result to the existence of solutions of the optimistic and pessimistic problems.
Section~\ref{sec:conclusions} concludes.

\section{Preliminaries}\label{sec:preliminaries}

We recall the main notions from vector optimization and set-valued analysis used throughout the paper. One may refer to \cite{Luc1989,AubinFrankowska2009,RockafellarWets1998} for more information.

Throughout the paper, all metric spaces are equipped with their given metric, and all products of finitely many metric spaces carry the product topology.
The Euclidean spaces are equipped with the respected Euclidean norms.
Our notation~$A \subset B$ includes the possibility that~$A = B$.

\subsection*{Ordering cones and generalized multi-criteria games}

Let~$C\subset\R^{d}$ be a cone, that is,~$tC\subset C$ for every~$t\geq 0$.
The cone~$C$ is \emph{convex} if~$C+C\subset C$, \emph{pointed} if~$C\cap(-C)=\{0\}$, and \emph{solid} if~$\interior C\neq\emptyset$.
The order induced by~$C$ is defined by
\begin{align}\label{eq:cone-order}
  a\preceq_{C} b \quad\Longleftrightarrow\quad b-a\in C\eqfinp
\end{align}
If~$\interior C\neq\emptyset$, the notation~$a\prec_{C} b$ means~$b-a\in\interior C$.

\begin{definition}[\cite{Luc1989}]
  \label{def:efficiency}
  Following the standard cone-order approach to vector optimization \cite{Luc1989}, let~$K$ be a nonempty set and let~$v:K\to\R^{d}$ be a vector-valued objective function.
  Under the \emph{maximization} convention, a strategy~$x^{\ast}\in K$ is called:
  \begin{itemize}
    \item a \emph{$C$-efficient point over~$K$} if there is no~$x\in K$ such that~$v(x)-v(x^{\ast})\in C\setminus\{0\}$;
    \item a \emph{weakly~$C$-efficient point over~$K$} if there is no~$x\in K$ such that~$v(x)-v(x^{\ast})\in\interior C$, where we assume that~$\interior C\neq\emptyset$;
    \item a \emph{ideal~$C$-efficient point over~$K$} if~$v(x)\preceq_{C} v(x^{\ast})$ for every~$x\in K$.
  \end{itemize}
  Denoted by~$\Eff_{C}(v;K)$,~$\wEff_{C}(v;K)$, and~$\IEff_{C}(v;K)$, respectively, the sets of all~$C$-efficient points, of all weakly~$C$-efficient points, and of all ideal~$C$-efficient points, over~$K$.
\end{definition}

If~$C$ is pointed, then~$\IEff_{C}(v;K)\subset\Eff_{C}(v;K)$, and~$\IEff_{C}(v;K)=\Eff_{C}(v;K)$ whenever~$\IEff_{C}(v;K)\neq\emptyset$.
Pointedness cannot be dropped here, since an ideal point~$x^{\ast}$ is not efficient as soon as~$v(x)-v(x^{\ast})\in C\cap(-C)\setminus\{0\}$ for some~$x\in K$.
If~$C$ is solid and~$C\neq\R^{d}$, then we also have~$\Eff_{C}(v;K)\subset\wEff_{C}(v;K)$.

Next, we describe formally the generalized multi-criteria game which is the central concept of our present study.
Let~$N=\{1,\ldots,\ell\}$ be a finite set of players.
For each~$\nu\in N$, let~$X_{\nu}$ be a metric space and let~$C_{\nu}\subset\R^{d_{\nu}}$ be an ordering cone.
We use the notation
\begin{align*}
  X:=\prod_{\nu\in N}X_{\nu}\eqsepv\qquad
  X_{-\nu}:=\prod_{\mu\in N\setminus\{\nu\}}X_{\mu}\eqsepv\qquad
  x=(x_{\nu},x_{-\nu})\eqfinp
\end{align*}

\begin{definition}\label{def:game}
  A generalized multi-criteria noncooperative game is specified by
  \begin{align*}
    \Gamma:=\bigl(X_{\nu},K_{\nu},u_{\nu},C_{\nu}\bigr)_{\nu\in N}\eqfinv
  \end{align*}
  where, for each player~$\nu\in N$,
  \begin{align*}
    K_{\nu}:X_{-\nu}\multimap X_{\nu}
    \eqsepand
    u_{\nu}:X\to\R^{d_{\nu}}
  \end{align*}
  are the constraint map and vector-valued payoff function, respectively.
  A profile~$x^{\ast}\in X$ is \emph{feasible} if~$x_{\nu}^{\ast}\in K_{\nu}(x_{-\nu}^{\ast})$ for all~$\nu\in N$.
  A feasible profile~$x^{\ast}$ is called
  \begin{itemize}
    \item an \emph{efficient Nash equilibrium} (ENE) of~$\Gamma$ if~$x_{\nu}^{\ast}\in \Eff_{C_{\nu}}
      \bigl(u_{\nu}(\cdot,x_{-\nu}^{\ast});K_{\nu}(x_{-\nu}^{\ast})\bigr)$
      for every~$\nu \in N$;
    \item a \emph{weakly efficient Nash equilibrium} (wENE) of~$\Gamma$ if~$x_{\nu}^{\ast}\in \wEff_{C_{\nu}}
      \bigl(u_{\nu}(\cdot,x_{-\nu}^{\ast});K_{\nu}(x_{-\nu}^{\ast})\bigr)$ for every $\nu\in N$; 
    \item an \emph{ideal Nash equilibrium} (INE) of~$\Gamma$ if~$x_{\nu}^{\ast}\in \IEff_{C_{\nu}}
      \bigl(u_{\nu}(\cdot,x_{-\nu}^{\ast});K_{\nu}(x_{-\nu}^{\ast})\bigr)$ for every $\nu\in N$.
  \end{itemize}
  The notion of a wENE is used only when every~$C_{\nu}$ are solid. 
  These equilibrium sets are denoted respectively by~$\ENE(\Gamma)$,~$\wENE(\Gamma)$, and~$\INE(\Gamma)$.
\end{definition}

\subsection*{Set-valued analysis and set convergence}

The notions and results in this section are standard within the literature (see \eg \cite{AubinFrankowska2009,RockafellarWets1998}).
For a metric space~$(X,d)$, let~$\mathcal{K}(X)$ denote the family of all nonempty compact subsets of~$X$.
For~$x\in X$ and nonempty~$A\subset X$, set~$d(x,A):=\inf_{a\in A}d(x,a)$.
The Hausdorff distance between~$A,B\in\mathcal{K}(X)$ is defined by
\begin{align}\label{eq:hausdorff}
  h(A,B):=\max\left\{
    \sup_{a\in A}d(a,B),\ \sup_{b\in B}d(b,A)
  \right\}\eqfinp
\end{align}
It is known that~$(\mathcal{K}(X), h)$ is a metric space, and it is complete if~$X$ is.

\begin{definition}
  \label{def:set-valued-continuity}
  Let~$X$ and~$Y$ be topological spaces and let~$\Phi:Y\multimap X$ have nonempty compact values.
  We say that~$\Phi$ is
  \begin{itemize}
    \item \emph{upper semicontinuous} (usc) at~$y\in Y$ if, for every open set~$O\subset X$ containing~$\Phi(y)$, there exists a neighborhood~$V$ of~$y$ such that~$\Phi(y')\subset O$ for all~$y'\in V$;
    \item \emph{lower semicontinuous} (lsc) at~$y\in Y$ if, for every open set~$O\subset X$ with~$\Phi(y)\cap O\neq\emptyset$, there exists a neighborhood~$V$ of~$y$ such that~$\Phi(y')\cap O\neq\emptyset$ for all~$y'\in V$;
    \item \emph{continuous} at~$y$ if it is both usc and lsc there.
  \end{itemize}
\end{definition}

The following criteria are stated with a compact codomain so that the closed-graph characterization of upper semicontinuity is valid.

\begin{proposition}
  \label{prop:sequential-set-valued}
  Let~$X$ be a \emph{compact} metric space,~$Y$ a metric space, and~$\Phi:Y\to\mathcal{K}(X)$.
  For any~$y\in Y$, the following hold:
  \begin{enumerate}[label=\upshape(\alph*)]
    \item\label{it:seq-usc} $\Phi$ is usc at~$y$ if and only if for every sequence~$y_{m}\to y$ and every sequence~$x_{m}\in \Phi(y_{m})$ satisfying~$x_{m}\to x$, one has~$x\in \Phi(y)$;
    \item\label{it:seq-lsc} $\Phi$ is lsc at~$y$ if and only if for every sequence~$y_{m}\to y$ and every~$x\in \Phi(y)$, there exist~$x_{m}\in \Phi(y_{m})$ with~$x_{m}\to x$;
    \item\label{it:seq-cont} $\Phi$ is continuous at~$y$ if and only if~$h(\Phi(y_{m}),\Phi(y))\to0$ for every sequence~$y_{m}\to y$.
  \end{enumerate}
  In particular, for compact~$X$, a compact-valued correspondence has a closed graph if and only if it is upper semicontinuous.
\end{proposition}

% [rev] Removed: Definition (Painleve--Kuratowski limits) and the
% proposition comparing them with Hausdorff convergence; neither was used.

\section{Stability of ideal Nash equilibria}\label{sec:stability}

We investigate the behavior of the ideal Nash equilibrium set of a parameterized generalized multi-criteria game under perturbations of its vector-valued payoff functions and constraint maps.
The equilibrium and cone-order conventions used here are those of Definitions~\ref{def:efficiency} and~\ref{def:game}
In particular, recall again that~$x\in\INE(\Gamma)$ if and only if~$x$ is feasible and~$u_{\nu}(x)-u_{\nu}(w_{\nu},x_{-\nu})\in C_{\nu}$ for all~$\nu\in N$ and~$w_{\nu}\in K_{\nu}(x_{-\nu})$.

This section is splitted into three subsections.
Starting with~\S\ref{subsec:parameters}, we introduce the definition and metric structure of our parameter space~$\mathcal{P}$ and prove a general result that the feasible profiles are upper semicontinuous.
Following this, the main stability result of ideal Nash equilibrium set is then discussed.
In~\S\ref{subsec:scalarization}, we point out that when the ordering cones are all convex, then the multi-criteria game reduces to a system of scalar games.
The common solution of these scalar games can then be used to reconstruct fully the ideal Nash equilibrium set of the original multi-criteria game.
Moreover, the stability of these scalar games imply the stability of the multi-criteria game.
We also present in this subsection~\S\ref{subsec:scalarization} an application to interval-valued games.
Finally, we present in~\S\ref{subsec:general-stability} the stability result in the general case where the cones are only assumed to be closed.

\subsection{The parameter space and feasible profiles}\label{subsec:parameters}

Throughout this section,~$N=\{1,\dots,\ell\}$ is finite and each~$X_{\nu}$ is a nonempty \emph{compact metric space} with metric~$d_{\nu}$.
Set~$X=\prod_{\nu\in N}X_{\nu}$ and retain the product notation~$X_{-\nu}$ from Section~\ref{sec:preliminaries}.
In this section, each ordering cone~$C_{\nu}\subset\R^{d_{\nu}}$ is a cone that need not be convex or pointed.
The notions of Section~\ref{sec:preliminaries} are defined for any cone.
Each result begins by stating what it assumes of the~$C_{\nu}$.
Pointedness of the~$C_{\nu}$ is not used anywhere in this section.
The results of this subsection do not involve the ordering cones.
No convexity of the strategy spaces is needed for the stability results in this section.

Let~$\mathcal{K}(X_{\nu})$ denote the nonempty compact subsets of~$X_{\nu}$, equipped with the Hausdorff distance~$h_{\nu}$ from \eqref{eq:hausdorff}.
Consider the parameter space
\begin{align}\label{eq:parameter-space}
  \mathcal{P}:=\prod_{\nu\in N}
  \Big(\mathscr{C}(X,\R^{d_{\nu}})\times
  \mathscr{C}\big(X_{-\nu},\mathcal{K}(X_{\nu})\big)\Big)\eqfinv
\end{align}
where~$\mathscr{C}(A,B)$ denotes the space of continuous maps from~$A$ to~$B$, with respect to their underlying topologies.
An element~$\lambda\in\mathcal{P}$ is written as~$\lambda=(u_{\nu}^{\lambda},K_{\nu}^{\lambda})_{\nu\in N}$, where~$u_{\nu}^{\lambda}:X\to\R^{d_{\nu}}$ is a vector payoff and~$K_{\nu}^{\lambda}:X_{-\nu}\multimap X_{\nu}$ is a continuous, nonempty compact-valued feasible correspondence.
The parameter~$\lambda$ specifies the game
\begin{align}\label{eq:parameter-game}
  \Gamma(\lambda):=\bigl(X_{\nu},K_{\nu}^{\lambda},
  u_{\nu}^{\lambda},C_{\nu}\bigr)_{\nu\in N}\eqfinp
\end{align}

Next, we define a distance function on~$\mathcal{P}$.
For~$\lambda,\eta\in\mathcal{P}$, define
\begin{align}\label{eq:parameter-metric}
  \begin{split}
    \rho(\lambda,\eta):={}&
    \sum_{\nu\in N}\sup_{x\in X}
    \|u_{\nu}^{\lambda}(x)-u_{\nu}^{\eta}(x)\| +\sum_{\nu\in N}\sup_{y\in X_{-\nu}}
    h_{\nu}\big(K_{\nu}^{\lambda}(y),K_{\nu}^{\eta}(y)\big)\eqfinp
  \end{split}
\end{align}

The following completeness result is an immediate consequence of the completeness of the primal data and the definition of~$\rho$.
Its proof is omitted.
\begin{proposition}
  \label{prop:ambient-complete}
  Under the assumptions above,~$(\mathcal{P},\rho)$ is a complete metric space.
\end{proposition}

To develop further, we define two set-valued maps~$F : \mathcal{P} \times X \multimap X$ and~$f : \mathcal{P} \multimap X$ by
\begin{align}\label{eq:joint-feasible}
  F(\lambda,x):=\prod_{\nu\in N}K_{\nu}^{\lambda}(x_{-\nu})
  \qquad ((\lambda,x)\in\mathcal{P}\times X)
\end{align}
and 
\begin{align}\label{eq:feasible-profile-set}
  f(\lambda):=\fix\big(F(\lambda,\cdot)\big)
  =\big\{x\in X\mid x_{\nu}\in K_{\nu}^{\lambda}(x_{-\nu})
  \ \text{for all }\nu\in N\big\}\eqfinp
\end{align}
A priori,~$f(\lambda)$ may be empty.
Write
\begin{align}\label{eq:ideal-equilibrium-set}
  S(\lambda):=\INE(\Gamma(\lambda)),\eqsepand
  \Lambda:= \dom S = \{\lambda\in\mathcal{P}\mid S(\lambda)\neq\emptyset\}\eqfinp
\end{align}
In particular,~$S(\lambda)\subset f(\lambda)$ and~$f(\lambda)\neq\emptyset$ for every~$\lambda\in\Lambda$.

\begin{lemma}\label{lem:feasibility}
  The following assertions hold:
  \begin{enumerate}[label=\upshape(\alph*)]
    \item\label{it:feas-closed} the graph~$\gr f$ is closed;
    \item\label{it:feas-compact} for each~$\lambda\in\mathcal{P}$,~$f(\lambda)$ is compact (possibly empty);
    \item\label{it:feas-usc} the restriction~$f:\Lambda\multimap X$ is upper semicontinuous and has nonempty compact values.
  \end{enumerate}
\end{lemma}
\begin{proof}
  \ref{it:feas-closed} Suppose~$\lambda_{m}\to\lambda$ in~$\rho$ and~$x^{m}\to x$ in~$X$, with~$x^{m}\in f(\lambda_{m})$.
  For each~$\nu$,
  \begin{align*}
&h_{\nu}\big(K_{\nu}^{\lambda_{m}}(x^{m}_{-\nu}),
K_{\nu}^{\lambda}(x_{-\nu})\big)\\
&\quad\leq
\sup_{y\in X_{-\nu}}h_{\nu}\big(K_{\nu}^{\lambda_{m}}(y),
K_{\nu}^{\lambda}(y)\big)
+h_{\nu}\big(K_{\nu}^{\lambda}(x^{m}_{-\nu}),
K_{\nu}^{\lambda}(x_{-\nu})\big)
\to0\eqfinp
  \end{align*}
  The first term tends to zero by \eqref{eq:parameter-metric} and the second by the Hausdorff continuity of~$K_{\nu}^{\lambda}$.
  Since~$x^{m}_{\nu}\in K_{\nu}^{\lambda_{m}}(x^{m}_{-\nu})$ and~$x_{\nu}^{m}\to x_{\nu}$, the limiting point satisfies~$x_{\nu}\in K_{\nu}^{\lambda}(x_{-\nu})$.
  This holds for every~$\nu$, so~$x\in f(\lambda)$.
  This shows that the graph is closed.
  \ref{it:feas-compact} By~\ref{it:feas-closed}, each value is a closed subset of compact~$X$ and is therefore compact.

  \ref{it:feas-usc} On~$\Lambda$ the values are nonempty, and the sequential upper-semicontinuity criterion of Proposition~\ref{prop:sequential-set-valued} applies (see also \cite{AubinFrankowska2009}).
\end{proof}

\subsection{Convex ordering cones: reduction to scalar games}
\label{subsec:scalarization}

When the ordering cones are convex, ideal Nash equilibria are exactly the common equilibria of a family of scalar generalized games, obtained by weighting each player's payoff with vectors of the dual cone.
We record this reduction and its consequences, and explain why it is not available for nonconvex cones.

The dual cone of~$C_{\nu}$ is~$C_{\nu}^{\ast}:=\{\xi\in\R^{d_{\nu}}\mid\langle\xi,y\rangle\geq0 \text{ for all }y\in C_{\nu}\}$.
We say that a nonempty set~$G_{\nu}\subset\R^{d_{\nu}}$ \emph{describes}~$C_{\nu}$ if
\begin{align*}
  C_{\nu}=\big\{y\in\R^{d_{\nu}}\mid\langle\xi,y\rangle\geq0
  \ \text{for all }\xi\in G_{\nu}\big\}\eqfinp
\end{align*}
By the bipolar theorem~\cite{RockafellarWets1998},~$C_{\nu}^{\ast}$ describes every closed convex cone~$C_{\nu}$.
If~$C_{\nu}=\{y \in \R^{d}\mid\Xi_{\nu} y\in\R_{+}^{k_{\nu}}\}$ is polyhedral, the rows~$\xi_{\nu,1},\dots,\xi_{\nu,k_{\nu}}$ of~$\Xi_{\nu}$ describe~$C_{\nu}$.
By the Minkowski--Weyl theorem~\cite{RockafellarWets1998}, every polyhedral convex cone has such a representation.
For~$C_{\nu}=\R_{+}^{d_{\nu}}$, the standard unit vectors~$e_{1},\dots,e_{d_{\nu}}$ describe~$C_{\nu}$.

For~$\lambda\in\mathcal{P}$ and a weight profile~$\xi=(\xi_{\nu})_{\nu\in N}\in\prod_{\nu\in N}\R^{d_{\nu}}$, let~$\Gamma_{\xi}(\lambda)$ be the scalar generalized game in which player~$\nu$ maximizes~$\langle\xi_{\nu},u_{\nu}^{\lambda}\rangle$ subject to~$K_{\nu}^{\lambda}$.
Its set of generalized Nash equilibria is
\begin{align*}
  \GNE\big(\Gamma_{\xi}(\lambda)\big):=
  \Big\{x\in f(\lambda)\mid
    \langle\xi_{\nu},u_{\nu}^{\lambda}(w_{\nu},x_{-\nu})\rangle
    \leq\langle\xi_{\nu},u_{\nu}^{\lambda}(x)\rangle
  \ \text{for all }\nu\in N,\ w_{\nu}\in K_{\nu}^{\lambda}(x_{-\nu})\Big\}\eqfinp
\end{align*}

Next, using the dual cone trick, we show that INE can be reformulated equivalently into a common Nash equilibrium of corresponding scalar generalized Nash equilibrium problems.

\begin{proposition}
  \label{prop:scalar-reduction}
  Let each~$C_{\nu}$ be a closed convex cone, and let~$G_{\nu}$ describe~$C_{\nu}$ for every~$\nu\in N$.
  Then, for every~$\lambda\in\mathcal{P}$,
  \begin{align}\label{eq:scalar-reduction}
    S(\lambda)=\bigcap_{\xi\in\prod_{\nu\in N}G_{\nu}}
    \GNE\bigl(\Gamma_{\xi}(\lambda)\bigr)\eqfinp
  \end{align}
  Moreover, the intersection may be taken over any family~$(\xi^{t})_{t\in T}\subset\prod_{\nu\in N}G_{\nu}$ such that~$\{\xi^{t}_{\nu}\mid t\in T\}=G_{\nu}$ for every~$\nu\in N$.
\end{proposition}
\begin{proof}
  Both sides are subsets of~$f(\lambda)$, so fix~$x\in f(\lambda)$.
  By Definitions~\ref{def:efficiency} and~\ref{def:game},~$x\in S(\lambda)$ if and only if~$u_{\nu}^{\lambda}(x)-u_{\nu}^{\lambda}(w_{\nu},x_{-\nu})\in C_{\nu}$ for all~$\nu\in N$ and~$w_{\nu}\in K_{\nu}^{\lambda}(x_{-\nu})$.
  Since~$G_{\nu}$ describes~$C_{\nu}$, this holds if and only if
  \begin{align*}
    \big\langle\xi_{\nu},\,u_{\nu}^{\lambda}(x)-u_{\nu}^{\lambda}(w_{\nu},x_{-\nu})\big\rangle
    \geq 0
    \quad\text{for all }\nu\in N,\ \xi_{\nu}\in G_{\nu},\
    w_{\nu}\in K_{\nu}^{\lambda}(x_{-\nu})\eqfinp
  \end{align*}
  On the other hand,~$x\in\GNE(\Gamma_{\xi^{t}}(\lambda))$ for every~$t\in T$ if and only if the same inequality holds with~$\xi^{t}_{\nu}$ in place of~$\xi_{\nu}$, for all~$t\in T$,~$\nu\in N$, and~$w_{\nu}\in K_{\nu}^{\lambda}(x_{-\nu})$.
  Since~$\{\xi^{t}_{\nu}\mid t\in T\}=G_{\nu}$ for every~$\nu$, the two conditions coincide.
  The full product~$\prod_{\nu\in N}G_{\nu}$ is such a family, which gives \eqref{eq:scalar-reduction}.
\end{proof}

We now combine Proposition~\ref{prop:scalar-reduction} with the closed-graph property of solution maps of parametric scalar generalized Nash games established by~\citet[Theorem~3.1]{AusselSvensson2018} (see also \cite[Theorem~3.3.4 and its proof]{AusselSvensson2020}, and \cite[Corollary~4.4]{LignolaMorgan1998} for strategy sets in reflexive Banach spaces).
The following lemma is extracted from the first step in the proof of~\citet[Theorem~4.2.4]{Salas2025}, which follows the proof of~\citet[Theorem~3.1]{AusselSvensson2018}.
Although the original proof was under the setting of a Euclidean space, the argument however uses only the sequential characterization of lower semicontinuity.
We include its proof for the sake of completeness.

\begin{lemma} \label{lem:gnep-closed}
  Let~$T$ and~$Y_{1},\dots,Y_{M}$ be metric spaces, and put~$Y:=\prod_{j=1}^{M}Y_{j}$.
  For each~$j$, let~$g_{j}:T\times Y_{j}\times Y_{-j}\to\R$ be continuous, and let~$Q_{j}:T\times Y_{-j}\multimap Y_{j}$ be lower semicontinuous with closed graph.
  For~$t\in T$, let~$E(t)$ be the set of generalized Nash equilibria of the game in which player~$j$ minimizes~$g_{j}(t,\cdot,y_{-j})$ over~$Q_{j}(t,y_{-j})$, that is,
  \begin{align*}
    \begin{aligned}
      E(t):=\big\{y\in Y\mid\ {}&y_{j}\in Q_{j}(t,y_{-j})\ \text{and}\
        g_{j}(t,y_{j},y_{-j})\leq g_{j}(t,z_{j},y_{-j})\\
                                &\text{for all }z_{j}\in Q_{j}(t,y_{-j})\text{ and }j=1,\dots,M\big\}\eqfinp
    \end{aligned}
  \end{align*}
  Then the set-valued map~$E:T\multimap Y$ has a closed graph.
\end{lemma}
\begin{proof}
  Let~$t_{k}\to t$ and~$y^{k}\in E(t_{k})$ with~$y^{k}\to y$, and fix~$j$.
  Since~$(t_{k},y^{k}_{-j},y^{k}_{j})$ lies in the closed graph of~$Q_{j}$, we get~$y_{j}\in Q_{j}(t,y_{-j})$.
  Given~$z_{j}\in Q_{j}(t,y_{-j})$, lower semicontinuity of~$Q_{j}$ yields~$z_{j}^{k}\in Q_{j}(t_{k},y^{k}_{-j})$, for all large~$k$, with~$z_{j}^{k}\to z_{j}$.
  Passing to the limit in~$g_{j}(t_{k},y^{k}_{j},y^{k}_{-j})\leq g_{j}(t_{k},z^{k}_{j},y^{k}_{-j})$ by continuity of~$g_{j}$ gives~$g_{j}(t,y_{j},y_{-j})\leq g_{j}(t,z_{j},y_{-j})$.
  Hence~$y\in E(t)$.
\end{proof}

The next lemma verifies the hypotheses of Lemma~\ref{lem:gnep-closed} for the data of~$\Gamma(\lambda)$.
Note that we do not need any assumption on the cones~$C_{\nu}$'s here.

\begin{lemma}
  \label{lem:data-continuity}
  For each~$\nu\in N$:
  \begin{enumerate}[label=\upshape(\alph*)]
    \item\label{it:data-payoff} the map~$(\lambda,x)\mapsto u_{\nu}^{\lambda}(x)$ is continuous on~$\mathcal{P}\times X$;
    \item\label{it:data-constraint} the correspondence~$(\lambda,y)\mapsto K_{\nu}^{\lambda}(y)$ from~$\mathcal{P}\times X_{-\nu}$ to~$X_{\nu}$ has a closed graph and is lower semicontinuous.
  \end{enumerate}
\end{lemma}
\begin{proof}
  \ref{it:data-payoff} If~$(\lambda_{m},x^{m})\to(\lambda,x)$, then~$\|u_{\nu}^{\lambda_{m}}(x^{m})-u_{\nu}^{\lambda}(x)\|\leq\rho(\lambda_{m},\lambda) +\|u_{\nu}^{\lambda}(x^{m})-u_{\nu}^{\lambda}(x)\|\to0$ by \eqref{eq:parameter-metric} and continuity of~$u_{\nu}^{\lambda}$.

  \ref{it:data-constraint} If~$(\lambda_{m},y^{m})\to(\lambda,y)$, the estimate in the proof of Lemma~\ref{lem:feasibility} gives~$h_{\nu}(K_{\nu}^{\lambda_{m}}(y^{m}),K_{\nu}^{\lambda}(y))\to0$.
  If~$w^{m}\in K_{\nu}^{\lambda_{m}}(y^{m})$ and~$w^{m}\to w$, then~$d_{\nu}(w,K_{\nu}^{\lambda}(y))=0$, so~$w\in K_{\nu}^{\lambda}(y)$ because this set is closed.
  Hence the graph is closed.
  Given~$w\in K_{\nu}^{\lambda}(y)$, choose~$w^{m}\in K_{\nu}^{\lambda_{m}}(y^{m})$ with~$d_{\nu}(w^{m},w)\leq h_{\nu}(K_{\nu}^{\lambda_{m}}(y^{m}),K_{\nu}^{\lambda}(y))$. 
  Then we have~$w^{m}\to w$, and Proposition~\ref{prop:sequential-set-valued} gives lower semicontinuity.
\end{proof}

We are now ready to present the main stability result in the case where all the ordering cones are convex, which follows from that of its scalar counterparts.

\begin{theorem}
  \label{thm:convex-stability}
  Let each~$C_{\nu}$ be a closed convex cone.
  Then the graph~$\gr S$ is closed.
  Consequently,~$\Lambda$ is closed in~$\mathcal{P}$,~$(\Lambda,\rho)$ is a complete metric space, and~$S:\Lambda\to\mathcal{K}(X)$ is upper semicontinuous with nonempty compact values.
\end{theorem}
\begin{proof}
  Fix~$\xi\in\prod_{\nu\in N}C_{\nu}^{\ast}$.
  In the scalar game~$\Gamma_{\xi}(\lambda)$ defined before Proposition~\ref{prop:scalar-reduction}, player~$\nu$ minimizes~$-\langle\xi_{\nu},u_{\nu}^{\lambda}\rangle$ subject to~$K_{\nu}^{\lambda}$.
  By Lemma~\ref{lem:data-continuity}\ref{it:data-payoff}, these objectives are continuous in~$(\lambda,x)$, and by Lemma~\ref{lem:data-continuity}\ref{it:data-constraint}, the constraint maps~$(\lambda,x_{-\nu})\mapsto K_{\nu}^{\lambda}(x_{-\nu})$ are lower semicontinuous with closed graph.
  Lemma~\ref{lem:gnep-closed}, applied with~$T=\mathcal{P}$,~$Y_{\nu}=X_{\nu}$,~$g_{\nu}(\lambda,x_{\nu},x_{-\nu})=-\langle\xi_{\nu},u_{\nu}^{\lambda}(x)\rangle$, and~$Q_{\nu}(\lambda,x_{-\nu})=K_{\nu}^{\lambda}(x_{-\nu})$, shows that the graph of~$\lambda\mapsto\GNE(\Gamma_{\xi}(\lambda))$ is closed for every~$\xi$.
  By Proposition~\ref{prop:scalar-reduction} with~$G_{\nu}=C_{\nu}^{\ast}$, the graph~$\gr S$ is the intersection of these closed graphs.

  Each~$S(\lambda)$ is then a closed subset of compact~$X$, hence compact.
  If~$\lambda_{m}\in\Lambda$ converge to~$\lambda$ and~$x^{m}\in S(\lambda_{m})$, a subsequence of~$(x^{m})$ converges to some~$x\in X$, and~$x\in S(\lambda)$ by the closed graph.
  Hence~$\Lambda$ is closed, and therefore complete by Proposition~\ref{prop:ambient-complete}.
  Upper semicontinuity of~$S:\Lambda\to\mathcal{K}(X)$ follows from Proposition~\ref{prop:sequential-set-valued}.
\end{proof}

\subsubsection*{Application to interval-valued games}
\label{sec:interval-valued-games}

An \emph{interval-valued optimization problem} (IOP) has an objective whose values are compact intervals \cite{Wu2007}.
Typically, an IOP takes advantage of the LU domination relation, denoted with~$\preceq_{\mathrm{LU}}$, between two compact intervals.
For any two compact intervals~$U:=[\underline{u},\overline{u}]$ and~$V:=[\underline{v},\overline{v}]$, we say
\begin{align*}
  U\preceq_{\mathrm{LU}}V
  \quad\Longleftrightarrow\quad
  \underline{u}\leq\underline{v}
  \ \text{and}\ 
  \overline{u}\leq\overline{v}
  \quad\Longleftrightarrow\quad
  (\underline{u},\overline{u})
  \preceq_{\R_{+}^{2}}
  (\underline{v},\overline{v})\eqfinp
\end{align*}
Thus the map~$[\underline{u},\overline{u}]\mapsto(\underline{u},\overline{u})$ embeds the compact intervals, ordered by~$\preceq_{\mathrm{LU}}$, into~$\R^{2}$ ordered by the closed, convex, pointed cone~$\R_{+}^{2}$.
Its image is the closed set~$\{(a,b)\in\R^{2}\mid a\leq b\}$.

For an interval-valued objective~$\mathcal{I}=[\underline{\varphi},\overline{\varphi}]$ with~$\underline{\varphi}\leq\overline{\varphi}$ on a feasible set~$A$,~\citet{Wu2007} studies minimization and calls a point LU-optimal if it is LU-nondominated.
In the maximization convention of this paper, the LU-nondominated points form~$\Eff_{\R_{+}^{2}}\big((\underline{\varphi},\overline{\varphi});A\big)$, and we call the points of~$\IEff_{\R_{+}^{2}}\big((\underline{\varphi},\overline{\varphi});A\big)$ \emph{ideal LU-solutions}.
The two sets coincide whenever an ideal LU-solution exists (recall the discussion after Definition~\ref{def:efficiency}).

In a \emph{generalized interval-valued game}, each player~$\nu\in N$ has a constraint map~$K_{\nu}:X_{-\nu}\multimap X_{\nu}$ and an interval-valued payoff~$U_{\nu}(x)=[\underline{u}_{\nu}(x),\overline{u}_{\nu}(x)]$ with~$\underline{u}_{\nu}\leq\overline{u}_{\nu}$ on~$X$.
A feasible profile~$x^{\ast}$ is an \emph{ideal Nash equilibrium} if, for every~$\nu\in N$,
\begin{align*}
  U_{\nu}(w_{\nu},x_{-\nu}^{\ast})\preceq_{\mathrm{LU}}U_{\nu}(x^{\ast})
  \qquad\text{for all }w_{\nu}\in K_{\nu}(x_{-\nu}^{\ast})\eqfinp
\end{align*}
By the embedding, this is exactly an INE in the sense of Definition~\ref{def:game} of the multi-criteria game with~$u_{\nu}=(\underline{u}_{\nu},\overline{u}_{\nu})$,~$d_{\nu}=2$, and~$C_{\nu}=\R_{+}^{2}$.

Accordingly, in the setting of Section~\ref{sec:stability} take~$d_{\nu}=2$ and~$C_{\nu}=\R_{+}^{2}$ for every~$\nu\in N$, and define
\begin{align}\label{eq:interval-setting}
  \mathcal{P}_{\mathrm{I}}:=\big\{\lambda\in\mathcal{P}\mid
    u_{\nu,1}^{\lambda}(x)\leq u_{\nu,2}^{\lambda}(x)
  \ \text{for all }\nu\in N\text{ and }x\in X\big\}\eqsepv
  \qquad
  \Lambda_{\mathrm{I}}:=\Lambda\cap\mathcal{P}_{\mathrm{I}}\eqfinp
\end{align}
Each~$\lambda\in\mathcal{P}_{\mathrm{I}}$ is an interval-valued game with~$U_{\nu}^{\lambda}=[u_{\nu,1}^{\lambda},u_{\nu,2}^{\lambda}]$, and~$S(\lambda)$ is its set of ideal Nash equilibria.

\begin{corollary}
  \label{cor:interval-stability}
  In the setting of~\eqref{eq:interval-setting}, the following hold:
  \begin{enumerate}[label=\upshape(\alph*)]
    \item\label{it:int-closed} $\mathcal{P}_{\mathrm{I}}$ and~$\Lambda_{\mathrm{I}}$ are closed in~$\mathcal{P}$, and~$(\Lambda_{\mathrm{I}},\rho)$ is a complete metric space;
    \item\label{it:int-usc} the set~$\{(\lambda,x)\in\mathcal{P}_{\mathrm{I}}\times X\mid x\in S(\lambda)\}$ is closed, and~$S:\Lambda_{\mathrm{I}}\to\mathcal{K}(X)$ is upper semicontinuous;
    \item\label{it:int-scalar} $S(\lambda)=\GNE\big(\Gamma_{\xi^{1}}(\lambda)\big)\cap\GNE\big(\Gamma_{\xi^{2}}(\lambda)\big)$ for every~$\lambda\in\mathcal{P}_{\mathrm{I}}$, where~$\xi^{i}_{\nu}:=e_{i}$ for all~$\nu\in N$ and~$i=1,2$.
  \end{enumerate}
\end{corollary}
\begin{proof}
  \ref{it:int-closed}--\ref{it:int-usc} Uniform limits preserve the inequalities~$u_{\nu,1}^{\lambda}\leq u_{\nu,2}^{\lambda}$, so~$\mathcal{P}_{\mathrm{I}}$ is closed.
  Both claims then follow from Theorem~\ref{thm:convex-stability} by intersecting~$\Lambda$ with~$\mathcal{P}_{\mathrm{I}}$ and~$\gr S$ with~$\mathcal{P}_{\mathrm{I}}\times X$.
  The completeness follows from Proposition~\ref{prop:ambient-complete}.

  \ref{it:int-scalar} This is Proposition~\ref{prop:scalar-reduction} with~$G_{\nu}=\{e_{1},e_{2}\}$ and the family~$(\xi^{1},\xi^{2})$.
\end{proof}

\subsection{Closed, possibly nonconvex, ordering cones}
\label{subsec:general-stability}

The results of this subsection hold for every closed ordering cone. 
The next theorem proves equilibrium stability directly, without requiring that ideal payoff vectors exist at arbitrary feasible profiles.

\begin{theorem}
  \label{thm:equilibrium-stability}
  Let each~$C_{\nu}$ be a closed cone.
  For the parameterized games \eqref{eq:parameter-game} with the metric \eqref{eq:parameter-metric}, the graph~$\gr S$ is closed, where~$S:\mathcal{P}\multimap X$ permits empty values.
  Consequently:
  \begin{enumerate}[label=\upshape(\alph*)]
    \item\label{it:stab-closed} the set~$\Lambda$ in \eqref{eq:ideal-equilibrium-set} is a closed subset of~$\mathcal{P}$, and~$(\Lambda,\rho)$ is a complete metric space;
    \item\label{it:stab-compact} $S(\lambda)$ is nonempty and compact for every~$\lambda\in\Lambda$;
    \item\label{it:stab-usc} $S:\Lambda\to\mathcal{K}(X)$ is upper semicontinuous.
  \end{enumerate}
\end{theorem}
\begin{proof}
  Let~$\lambda_{m}\to\lambda$ in~$\mathcal{P}$ and~$x^{m}\to x$ in~$X$, with~$x^{m}\in S(\lambda_{m})$.
  By Lemma~\ref{lem:feasibility}\ref{it:feas-closed},~$x\in f(\lambda)$.
  Fix a player~$\nu\in N$ and an arbitrary feasible deviation~$w_{\nu}\in K_{\nu}^{\lambda}(x_{-\nu})$.
  As in the proof of Lemma~\ref{lem:feasibility},
  \begin{align*}
    h_{\nu}\big(K_{\nu}^{\lambda_{m}}(x_{-\nu}^{m}),
    K_{\nu}^{\lambda}(x_{-\nu})\big)\to0\eqfinp
  \end{align*}
  Therefore one may choose~$w_{\nu}^{m}\in K_{\nu}^{\lambda_{m}}(x_{-\nu}^{m})$ so that
  \begin{align*}
    d(w_{\nu}^{m}, w_{\nu}) = d(w_{\nu}, K_{\nu}^{\lambda_{m}}(x_{-\nu}^{m})) \leq h(K_{\nu}^{\lambda}(x_{-\nu}), K_{\nu}^{\lambda_{m}}(x_{-\nu}^{m})) \to 0
    \eqfinp
  \end{align*}
  This gives~$w_{\nu}^{m}\to w_{\nu}$.
  Since~$x^{m}$ is an INE of~$\Gamma(\lambda_{m})$, Definition~\ref{def:game} gives
  \begin{align}\label{eq:approx-ideal-order}
    u_{\nu}^{\lambda_{m}}(x^{m})-
    u_{\nu}^{\lambda_{m}}(w_{\nu}^{m},x_{-\nu}^{m})\in C_{\nu}\eqfinp
  \end{align}
  Uniform convergence of~$u_{\nu}^{\lambda_{m}}$ to~$u_{\nu}^{\lambda}$ and continuity of~$u_{\nu}^{\lambda}$ imply
  \begin{align*}
    u_{\nu}^{\lambda_{m}}(x^{m})&\to u_{\nu}^{\lambda}(x)\eqfinv\\
    u_{\nu}^{\lambda_{m}}(w_{\nu}^{m},x_{-\nu}^{m})
                                &\to u_{\nu}^{\lambda}(w_{\nu},x_{-\nu})\eqfinp
  \end{align*}
  As~$C_{\nu}$ is closed, passing to the limit in \eqref{eq:approx-ideal-order} yields
  \begin{align*}
    u_{\nu}^{\lambda}(x)-u_{\nu}^{\lambda}(w_{\nu},x_{-\nu})
    \in C_{\nu}\eqfinp
  \end{align*}
  Thus~$u_{\nu}^{\lambda}(w_{\nu},x_{-\nu}) \preceq_{C_{\nu}}u_{\nu}^{\lambda}(x)$.
  Because both~$\nu$ and its feasible deviation~$w_{\nu}$ were arbitrary, we have shown that~$x\in S(\lambda)$.
  This proves the graph of~$S$ is closed.

  For each~$\lambda\in\mathcal{P}$,~$S(\lambda)$ is a closed subset of compact~$X$ and hence compact.
  To see that~$\Lambda$ is closed, let~$\lambda_{m}\in\Lambda$ converge to~$\lambda\in\mathcal{P}$ and select~$x^{m}\in S(\lambda_{m})$.
  Compactness of~$X$ gives a subsequence~$x^{m_{k}}\to x$, and closedness of~$\gr S$ yields~$x\in S(\lambda)$.
  Therefore~$\lambda\in\Lambda$.
  Proposition~\ref{prop:ambient-complete} then implies completeness of~$(\Lambda,\rho)$.
  Finally, the sequential criterion in Proposition~\ref{prop:sequential-set-valued} shows that the nonempty compact-valued restriction~$S:\Lambda\to\mathcal{K}(X)$ is upper semicontinuous.
\end{proof}

Since~$(\Lambda,\rho)$ is complete, the upper semicontinuity in Theorem~\ref{thm:equilibrium-stability} upgrades to continuity at most parameters, in the sense of Baire category.
This is also known as the generic stability result.

\begin{corollary}
  \label{cor:generic-stability}
  Let each~$C_{\nu}$ be a closed cone.
  There is a dense residual subset~$\Lambda^{\ast}\subset\Lambda$ such that~$S:\Lambda\to\mathcal{K}(X)$ is continuous at every~$\lambda\in\Lambda^{\ast}$.
  Equivalently, for every~$\lambda\in\Lambda^{\ast}$ and every sequence~$\lambda_{m}\to\lambda$ in~$\Lambda$,~$h\big(S(\lambda_{m}),S(\lambda)\big)\to0$, where~$h$ is the Hausdorff distance on~$\mathcal{K}(X)$.
\end{corollary}
\begin{proof}
  By Theorem~\ref{thm:equilibrium-stability},~$(\Lambda,\rho)$ is a complete metric space, hence a Baire space, and~$S:\Lambda\to\mathcal{K}(X)$ is upper semicontinuous with nonempty compact values.
  Fort's theorem \cite{Fort1951} implies that the set~$\Lambda^{\ast}$ of points at which~$S$ is lower semicontinuous is residual in~$\Lambda$.
  It is dense by the Baire category theorem.
  At every~$\lambda\in\Lambda^{\ast}$,~$S$ is both upper and lower semicontinuous, and Proposition~\ref{prop:sequential-set-valued}\ref{it:seq-cont} gives the Hausdorff formulation.
\end{proof}

\section{Single-leader--multi-follower problems}
\label{sec:single-leader-multi-follower}

A single-leader--multi-follower game models the hierarchical decision process between a single leader and a group of followers~\cite{orig26,orig23,orig19}.
The followers make their decision simultaneously without communication, knowing the decision made by the leader.
This means that the followers' decision problem is a noncooperative game parametrized by the leader's variable~\cite{AusselSvensson2018,AusselSvensson2020}.
In this section we prove existence results for single-leader--multi-follower problems where each follower has vector-valued payoffs, which makes our investigation of generalized multi-criteria games from Section~\ref{sec:stability} relevant as the lower-level setting.
We consider both the optimistic and pessimistic approaches of single-leader--multi-follower problems~\cite{LoridanMorgan1989,Dempe2002,orig12,LamparielloSagratellaStein2019}.

We apply the upper semicontinuity of the ideal Nash equilibrium correspondence from Theorem~\ref{thm:equilibrium-stability} to a single-leader--multi-follower problem.
In the optimistic formulation, the leader minimizes its objective over both its admissible decisions and the followers' equilibrium responses (see the hierarchical optimization framework in \cite{Dempe2002,LoridanMorgan1989}).
In the pessimistic formulation, the leader instead guards against the least favorable follower equilibrium.
We divide this section into two subsections where \S\ref{subsec:optimistic} treats the optimistic formulation and~\S\ref{subsec:pessimistic} the pessimistic one.
Both rest on the closed-graph and semicontinuity properties of Theorem~\ref{thm:equilibrium-stability}.

Let~$Z_{0}$ be a nonempty compact metric space denoting the leader decisions.
Let~$\lambda:Z_{0}\to\mathcal{P}$ be a continuous map associating a follower game with each leader decision, where~$\mathcal{P}$ is the parameter space defined in \eqref{eq:parameter-space}.
For~$z\in Z_{0}$, write
\begin{align}\label{eq:leader-follower-identification}
  u_{\nu}(x,z):=u_{\nu}^{\lambda(z)}(x)\eqsepv\qquad
  K_{\nu}(x_{-\nu},z):=K_{\nu}^{\lambda(z)}(x_{-\nu})\eqsepv
  \qquad(\nu\in N)\eqfinp
\end{align}
The followers' ideal Nash equilibrium set at~$z$ is~$S(\lambda(z))=\INE(\Gamma(\lambda(z)))$.
Finally, let~$\theta:Z_{0}\times X\to\R$ denote the leader's scalar objective function.

\subsection{Optimistic approach}
\label{subsec:optimistic}

In the (standard) optimistic bilevel formulation \cite{Dempe2002,LoridanMorgan1989}, the leader solves
\begin{align}\label{eq:slmfvr-program}
  \begin{aligned}
    \min_{z\in Z_{0},\,x\in X}\quad &\theta(z,x)\\
    \text{subject to}\quad &x\in S(\lambda(z))\eqfinv
  \end{aligned}
\end{align}
where, as in Definition~\ref{def:game}, the follower condition~$x\in S(\lambda(z))$ means that, for every~$\nu\in N$,
\begin{align}\label{eq:slmfvr-follower-conditions}
  \left\{\begin{aligned}
    &x_{\nu}\in K_{\nu}(x_{-\nu},z),\\
    &u_{\nu}(w_{\nu},x_{-\nu},z)
    \preceq_{C_{\nu}}u_{\nu}(x,z)
    \quad\text{for all }w_{\nu}\in K_{\nu}(x_{-\nu},z).
  \end{aligned}\right.
\end{align}

The existence result below follows the scheme of~\citet[Theorem~3.1]{AusselSvensson2018},~\cite[Theorem~3.3.4]{AusselSvensson2020} for single-leader--multi-follower games with scalar followers.
Here the followers have vector-valued payoffs and play ideal Nash equilibria, which depends on the leader's decision via~$\lambda$, with respect to closed (possibly nonconvex) cones.

\begin{theorem}
  \label{thm:optimistic-existence}
  Assume the hypotheses of Theorem~\ref{thm:equilibrium-stability}.
  Suppose that:
  \begin{enumerate}[label=\upshape(\roman*), widest=ii]
    \item\label{it:opt-lambda} $Z_{0}$ is a nonempty compact metric space and~$\lambda:Z_{0}\to\Lambda$ is continuous, where~$\Lambda$ is defined in \eqref{eq:ideal-equilibrium-set};
    \item\label{it:opt-theta} $\theta:Z_{0}\times X\to\R$ is lower semicontinuous.
  \end{enumerate}
  Then problem~\eqref{eq:slmfvr-program} has a solution.
\end{theorem}
\begin{proof}
  Since~$\lambda:Z_{0}\to\Lambda$ is continuous and~$S:\Lambda\to\mathcal{K}(X)$ is upper semicontinuous by Theorem~\ref{thm:equilibrium-stability}, the composite map~$S\circ\lambda:Z_{0}\to\mathcal{K}(X)$ is upper semicontinuous with nonempty compact values.
  By Proposition~\ref{prop:sequential-set-valued}, its graph
  \begin{align}\label{eq:slmfvr-feasible-graph}
    \mathcal{F}:=\gr(S\circ\lambda)=\bigl\{(z,x)\in Z_{0}\times X\mid x\in S(\lambda(z))\bigr\}
  \end{align}
  is closed in the compact space~$Z_{0}\times X$, hence compact, and it is nonempty because~$Z_{0}\neq\emptyset$ and the values of~$S\circ\lambda$ are nonempty.
  By the Weierstrass theorem, the lower semicontinuous function~$\theta$ attains its minimum on~$\mathcal{F}$ at some~$(z^{\ast},x^{\ast})$, which solves problem~\eqref{eq:slmfvr-program}.
\end{proof}

\begin{remark}
  \label{rem:optimistic-scope}
  The condition~$\lambda(Z_{0})\subset\Lambda$ ensures that~$S(\lambda(z))\ne\emptyset$ for all~$z\in Z_{0}$.
  It is not a consequence of compactness and continuity alone, as one could observe simply that maximizing~$(t,1-t)$ over~$t\in[0,1]$ with respect to the orthant~$\R_{+}^{2}$ has no ideal solution.
\end{remark}

\begin{remark}
  \label{rem:slmfvr-minimal-feasibility}
  The proof above also works if~$\lambda:Z_{0}\to\mathcal{P}$ is continuous (without requiring~$\lambda(Z_{0})\subset\Lambda$), provided that~$\lambda(z_{0})\in\Lambda$ for some~$z_{0}\in Z_{0}$, that is, some leader decision admits a follower equilibrium.
  Indeed,~$S\circ\lambda$ may then have empty values, but~$\mathcal{F}=\gr(S\circ\lambda)$ is still closed, because~$S$ has a closed graph (Theorem~\ref{thm:equilibrium-stability}) and~$\lambda$ is continuous.
  That is, any~$x_{0}\in S(\lambda(z_{0}))$ gives~$(z_{0},x_{0})\in\mathcal{F}$.
  This variant does not assert the existence of follower equilibria at every leader decision.
\end{remark}

\subsection{Pessimistic solutions}
\label{subsec:pessimistic}

Throughout this subsection we keep hypothesis~\ref{it:opt-lambda} of Theorem~\ref{thm:optimistic-existence} and strengthen~\ref{it:opt-theta} to continuity of~$\theta:Z_{0}\times X\to\R$.
For~$z\in Z_{0}$, the leader's pessimistic value at~$z$ is
\begin{align}\label{eq:pessimistic-value}
  \psi(z):=\sup_{x\in S(\lambda(z))}\theta(z,x)\eqfinp
\end{align}
The supremum is attained, i.e., it is a maximum, because~$S(\lambda(z))$ is nonempty and compact, since~$\lambda(z)\in\Lambda$ (Theorem~\ref{thm:equilibrium-stability}), and~$\theta$ is continuous.
The pessimistic problem is given as
\begin{align}\label{eq:pessimistic-program}
  \min_{z\in Z_{0}}\psi(z)\eqsepv
  \qquad
  v:=\inf_{z\in Z_{0}}\psi(z)\eqfinp
\end{align}
Note that here, the followers' responses are captured within the function~$\psi$ and do not appear explicitly in~\eqref{eq:pessimistic-program}.
Unlike the optimistic problem~\eqref{eq:slmfvr-program}, a solution of~\eqref{eq:pessimistic-program} is a one-sided strategy~$z^{\ast}\in Z_{0}$ of the leader.
Similar to the optimistic case, we call~$(z^{\ast},x^{\ast})\in Z_{0}\times X$ a \emph{paired solution} of~\eqref{eq:pessimistic-program} if~$x^{\ast}\in S(\lambda(z^{\ast}))$ and~$\theta(z^{\ast},x^{\ast})=\psi(z^{\ast})=v$.
When the supremum in~\eqref{eq:pessimistic-value} is attained for any~$z \in Z_{0}$, a decision~$z^{\ast}$ solves~\eqref{eq:pessimistic-program} if and only if~$(z^{\ast},x^{\ast})$ is a paired solution for some~$x^{\ast}$, and such~$x^{\ast}$ are exactly the maximizers in~\eqref{eq:pessimistic-value} at~$z^{\ast}$.

In the following theorem, we show that~$v$ is always attained on~$\gr(S\circ\lambda)$, and that, under lower semicontinuity of~$S$, the attaining pair is a paired solution.

\begin{theorem}
  \label{thm:pessimistic-existence}
  Assume the hypotheses of Theorem~\ref{thm:equilibrium-stability} and hypothesis~\ref{it:opt-lambda} of Theorem~\ref{thm:optimistic-existence}, and let~$\theta:Z_{0}\times X\to\R$ be continuous.
  Then
  \begin{enumerate}[label=\upshape(\alph*)]
    \item\label{it:pes-value} the pessimistic value~$v$ is finite and attained on~$\gr(S\circ\lambda)$, that is, there exist~$z^{\ast}\in Z_{0}$ and~$x^{\ast}\in S(\lambda(z^{\ast}))$ with~$\theta(z^{\ast},x^{\ast})=v$, where~$z^{\ast}$ is the limit of a minimizing sequence of~$\psi$ on~$Z_{0}$;
    \item\label{it:pes-strong} if~$z^{\ast}\in Z_{0}$ is the limit of a minimizing sequence of~$\psi$ on~$Z_{0}$ and~$S:\Lambda\to\mathcal{K}(X)$ is lower semicontinuous at~$\lambda(z^{\ast})$, then~$z^{\ast}$ solves problem~\eqref{eq:pessimistic-program}.
      If this~$z^{\ast}$ is the one in~\ref{it:pes-value}, then~$(z^{\ast},x^{\ast})$ there is a paired solution.
      In particular, this holds whenever~$\lambda(z^{\ast})\in\Lambda^{\ast}$, the dense residual set of Corollary~\ref{cor:generic-stability}.
  \end{enumerate}
\end{theorem}
\begin{proof}
  \ref{it:pes-value} As in the proof of Theorem~\ref{thm:optimistic-existence}, the graph~$\mathcal{F}=\gr(S\circ\lambda)$ in \eqref{eq:slmfvr-feasible-graph} is nonempty and compact, and~$v$ is finite because~$\min_{Z_{0}\times X}\theta\leq\psi\leq\max_{Z_{0}\times X}\theta$.
  Choose~$(z_{m},x^{m})\in\mathcal{F}$ with~$\theta(z_{m},x^{m})=\psi(z_{m})\to v$.
  By compactness of~$\mathcal{F}$, a subsequence converges to some~$(z^{\ast},x^{\ast})\in\mathcal{F}$, and continuity of~$\theta$ gives~$\theta(z^{\ast},x^{\ast})=v$.

  \ref{it:pes-strong} Let~$z_{m}\to z^{\ast}$ with~$\psi(z_{m})\to v$, and let~$\hat{x}\in S(\lambda(z^{\ast}))$ satisfy~$\theta(z^{\ast},\hat{x})=\psi(z^{\ast})$.
  Since~$\lambda$ is continuous and~$S$ is lower semicontinuous at~$\lambda(z^{\ast})$, Proposition~\ref{prop:sequential-set-valued} gives~$\hat{x}^{m}\in S(\lambda(z_{m}))$ with~$\hat{x}^{m}\to\hat{x}$, so that
  \begin{align*}
    v=\lim_{m\to\infty}\psi(z_{m})\geq\lim_{m\to\infty}\theta(z_{m},\hat{x}^{m})=\psi(z^{\ast})\geq v\eqfinp
  \end{align*}
  For the pair of~\ref{it:pes-value}, this gives~$\theta(z^{\ast},x^{\ast})=v=\psi(z^{\ast})$.
  At every~$\lambda\in\Lambda^{\ast}$, the correspondence~$S$ is continuous by Corollary~\ref{cor:generic-stability}, and in particular lower semicontinuous.
\end{proof}

\begin{remark}
  \label{rem:pessimistic-gap}
  The pair~$(z^{\ast},x^{\ast})$ in part~\ref{it:pes-value} need not be a paired solution, and~$z^{\ast}$ need not solve~\eqref{eq:pessimistic-program}.
  Indeed, we only have~$v=\theta(z^{\ast},x^{\ast})\leq\psi(z^{\ast})$ and the inequality may be strict.
  This is due to the fact that part~\ref{it:pes-value} uses only the closedness of~$\gr(S\circ\lambda)$.
  Lower semicontinuity of~$S$ at~$\lambda(z^{\ast})$ gives the missing bound~$\psi(z^{\ast})\leq v$, which closes this gap in part~\ref{it:pes-strong}.

  In fact, this is a known difficulty of pessimistic bilevel problems that they may fail to have a solution even under compactness and continuity assumptions, since the leader's pessimistic value function is in general only upper semicontinuous (see~\citet{LamparielloSagratellaStein2019}).
  Theorem~\ref{thm:pessimistic-existence} reflects this issue.
  For completeness, the following example illustrates this situation in the scalar setting. 
\end{remark}

\begin{example}
  \label{ex:pessimistic-gap}
  Take one follower with~$X_{1}=[0,1]$,~$d_{1}=1$,~$C_{1}=\R_{+}$, and~$K_{1}^{\lambda(z)}\equiv[0,1]$, and let~$Z_{0}=[0,1]$,~$u_{1}^{\lambda(z)}(x)=zx$, and~$\theta(z,x)=z-x$.
  The map~$\lambda$ is continuous, since~$\sup_{x\in[0,1]}|zx-z'x|=|z-z'|$.
  Moreover~$S(\lambda(z))=\{1\}$ for~$z>0$, and~$S(\lambda(0))=[0,1]$, so all hypotheses of Theorem~\ref{thm:pessimistic-existence} hold.
  Then~$\psi(z)=z-1$ for~$z>0$ and~$\psi(0)=\max_{x\in[0,1]}(-x)=0$, so~$v=-1$ is not attained by~$\psi$ and problem~\eqref{eq:pessimistic-program} has no solution.
  The value~$v=\theta(0,1)$ is attained on~$\gr(S\circ\lambda)$ as in Theorem~\ref{thm:pessimistic-existence}\ref{it:pes-value}, but~$(0,1)$ is not a paired solution since~$\psi(0)=0$, and~$S$ is not lower semicontinuous at~$\lambda(0)$.
\end{example}

%%%%%%%%%%%%%%%%%%%%%%%%%%%%%%%%%%%%%%%%%%%%%%%%%%%%%%%%%%%%%%%%%%%%%%%%%%%%%%%%%
\section{Conclusions}\label{sec:conclusions}

We have investigated the stability of ideal Nash equilibria in generalized multi-criteria games with vector-valued payoffs and strategy-dependent feasible sets.
The players' payoffs are compared through closed ordering cones, which need not be convex or pointed, preserving the vector nature of the equilibrium conditions without scalarizing the individual payoff functions (cf.~\cite{Luc1989}).
Under compactness of the strategy spaces and the continuity assumptions encoded by the parameter metric \eqref{eq:parameter-metric}, we proved that the ideal Nash equilibrium correspondence has a closed graph.
Consequently, Theorem~\ref{thm:equilibrium-stability} establishes its upper semicontinuity and compact-valuedness on the parameter subspace~$\Lambda$ of games admitting ideal Nash equilibria.
Since~$\Lambda$ is complete, Corollary~\ref{cor:generic-stability} shows that the correspondence is in fact continuous on a dense residual subset of~$\Lambda$.

For convex cones, ideal Nash equilibria are exactly the common equilibria of scalar games weighted by dual-cone vectors (Proposition~\ref{prop:scalar-reduction}), so part of the stability analysis can also be obtained from the scalar theory.
For closed nonconvex cones, no such reduction exists, and the direct arguments are needed.
Identifying the LU order on intervals with the~$\R_{+}^{2}$ order, we obtained the same closed-graph and upper semicontinuity results for interval-valued games, whose ideal Nash equilibria are exactly the common generalized Nash equilibria of the two scalar games with payoffs~$u_{\nu,1}$ and~$u_{\nu,2}$ (Corollary~\ref{cor:interval-stability}).

We then applied the stability result to an optimistic single-leader--multi-follower problem in which the followers' game~$\lambda(z)$ is determined by the leader's decision~$z$.
Assuming that the leader's decision space is compact, the parameter selection map is continuous with values in~$\Lambda$, and the leader's objective function is lower semicontinuous, Theorem~\ref{thm:optimistic-existence} guarantees a solution of the optimistic problem.
The argument uses compactness of the joint leader--follower feasible set and the Weierstrass minimum principle, as in the standard optimistic bilevel framework (cf.~\cite{Dempe2002,LoridanMorgan1989}).
For the pessimistic formulation with a continuous leader objective, Theorem~\ref{thm:pessimistic-existence} uses the same closed graph to show that the pessimistic value is attained on~$\gr(S\circ\lambda)$, and the corresponding leader decision solves the pessimistic problem when the equilibrium correspondence is lower semicontinuous at the limiting parameter, for instance on the dense residual set of Corollary~\ref{cor:generic-stability}.
Example~\ref{ex:pessimistic-gap} shows that this condition cannot be dropped.

\footnotesize

\section*{Acknowledgments}

The authors acknowledge the financial support provided by the Center of Excellence in Theoretical and Computational Science (TaCS-CoE), Faculty of Science, King Mongkut's University of Technology Thonburi (KMUTT), Bangkok, Thailand.
The first author appreciates the support provided by Petchra Pra Jom Klao Ph.D. Research Scholarship of KMUTT through grant no.~17/2564.

The authors also acknowledge the use of the large language model Claude Opus 5.5 (Anthropic) for the presentation, proof checking, grammatical corrections and reference recommendation.
None of the results were derived by the AI tools.
The human authors have checked and approved the final manuscript, and take full responsibility for its content.

\bibliography{References.bib}

% \bibliography{slmfgafe.bib}
% \bibliographystyle{abbrvnat}

\end{document}